\documentclass[11pt]{article}
\usepackage{a4wide}
\usepackage[utf8]{inputenc}
\usepackage[T1]{fontenc}
\usepackage{amsmath, amssymb, amsthm, latexsym, color, textcomp, mathtools}
\usepackage{paralist}
\usepackage[ngerman, english]{babel}
\usepackage[T1]{fontenc}
\usepackage{lmodern}
\usepackage[colorlinks=true,linkcolor=blue,citecolor=magenta,urlcolor=blue]{hyperref}
\usepackage{tikz}
\usetikzlibrary{arrows}

\newcommand{\R}{\mathbb{R}}
\newcommand{\Z}{\mathbb{Z}}  
\newcommand\hz[1]{#1\hphantom{0}}

\newcommand\HZ{\hphantom{0}}
\newcommand{\F}[3]{f(\mathcal{#1}^{#2}_{#3})}
\newcommand{\Fl}[3]{f\hspace{-1pt}\ell(\mathcal{#1}^{#2}_{#3})}
\newcommand{\SP}[3]{\mathcal{#1}^{#2}_{#3}}

\newcommand{\pos}{\operatorname{pos}}

\theoremstyle{plain}
\newtheorem{theorem}{Theorem}[section]
\newtheorem{prop}[theorem]{Proposition}

\newtheorem{conjecture*}{Conjecture}

\theoremstyle{definition}
\newtheorem{alg}[theorem]{Algorithm}

\theoremstyle{remark}

\title{A flag vector of a 3-sphere that is not\\ the flag vector of a 4-polytope%
\footnote{The first author was funded by DFG through the RTG \emph{Methods for Discrete Structures}. 
Research by the second author was supported by the DFG Collaborative Research Center TRR~109 ``Discretization in Geometry and Dynamics''.}}
\author{Philip Brinkmann\\ 
Institut f\"ur Mathematik, FU Berlin\\Arnimallee 2\\14195 Berlin, Germany\\
\url{philip.brinkmann@fu-berlin.de}
\and
G\"unter M.~Ziegler\\
Institut f\"ur Mathematik, FU Berlin\\Arnimallee 2\\14195 Berlin, Germany\\
\url{ziegler@math.fu-berlin.de}}
\date{{\small November 19, 2015}}

\begin{document}
    
    \maketitle
    
\begin{abstract}
        We present a first example of a flag vector of a polyhedral sphere
        that is not the flag vector of any polytope. 
        Namely, there is a unique
        $3$-sphere with the parameters $(f_0,f_1,f_2,f_3;f_{02})=(12,40,40,12;120)$,
        but this sphere is not realizable by a convex $4$-polytope.
		
        The $3$-sphere, which is $2$-simple and $2$-simplicial, 
        was found by Werner (2009); we present results of a computer enumeration which
        imply that the sphere with these parameters is unique. 
        We prove that it is non-polytopal in two ways: First, we show that it has no oriented matroid,
        and thus it is not realizable; this proof was found by computer, but can be verified by hand.
        The second proof is again a computer-based oriented matroid proof and shows that for exactly
        one of the facets this sphere does not even have a diagram based on this facet. 
		Using the non-polytopality, we finally prove that the sphere
		is not even embeddable as a polytopal complex.
\end{abstract}

\section{Introduction} 

A lot of work has gone into the characterization of the sets of all $f$-vectors $(f_0,f_1,\dots,f_{d-1})$ of convex $d$-polytopes, which we
denote by $\F{P}{d}{}$. Already in 1906, Steinitz \cite{Steinitz} characterized the $f$-vectors of $3$-polytopes as
\[
\F{P}{3}{}=\{(f_0,f_1,f_2)\in\Z^3: f_0-f_1+f_2=2,\ f_2\le 2f_0-4,\ f_0\le 2f_2-4\}.
\]
Grünbaum \cite[p.~131]{Gruenbaum} showed that the affine hull of $\F{P}{d}{}$ is a hyperplane in~$\R^d$, determined by
the Euler equation.
Moreover, the $f$-vectors of the simplicial polytopes, which we denote by $\F{P}{d}{s}$, were characterized completely
by McMullen's “$g$-conjecture” \cite{McMullenGConj}, 
as proven by Billera \& Lee \cite{BilleraLeeGTHM2} and Stanley~\cite{Sta3}.

The flag vector of a polytope, introduced by Bayer \& Billera \cite{BaBi} as the
“extended $f$-vector,” in general contains considerably more combinatorial information than the $f$-vector.
For example, the dimension of the affine hull of the set of 
flag vectors of $d$-polytopes, which we denote by  $\Fl{P}{d}{}$, is a Fibonacci number minus one \cite{BaBi}, and thus grows
exponentially with $d$.
On the other hand, within some important classes of polytopes the $f$-vector determines the flag vector;
for example, this holds for $3$-polytopes, as well as for simplicial polytopes (and thus for simple polytopes,
by duality).
Thus the results just quoted also characterize the flag vector sets
$\Fl{P}{3}{}$ and $\Fl{P}{d}{s}$.

While the $f$-vectors and the flag vectors of $3$-polytopes are thus completely understood,
for $4$-dimensional polytopes our knowledge about 
the sets of $f$-vectors and of flag vectors is rather incomplete.
The $2$-dimensional coordinate projections of the $3$-dimensional set $\F{P}{4}{}$ have been determined completely, 
see \cite[Sect.~10.4 and p.~198c]{Gruenbaum}.
A systematic study of the $4$-dimensional set $\Fl{P}{4}{}$ was started by Bayer \cite{bayer}; see 
also \cite{HoeppnerZie} and \cite{Ziegler03,ZieglerUtah}.
But there are still open questions, e.g.~those related to the \emph{fatness}. While there are $3$-spheres of arbitrarily large fatness, 
this is not known for $4$-polytopes \cite{ZieglerUtah}. On the other hand, there is a lower bound for the fatness of $4$-polytopes, 
but we do not know whether this also holds for $3$-spheres~\cite{ZieglerUtah}.

A $d$-\emph{sphere} is a regular CW $d$-sphere with the intersection property (i.e.~any two faces intersect
in a common face). We call a $d$-sphere \emph{polyhedral} if in addition it is PL and all facets are 
polytopal. However, all $3$-spheres are PL, since all $3$-manifolds are \cite[Sec.~36]{Moise77}, and all
facets ($2$-spheres) are polytopal by Steinitz' Theorem \cite{SteinitzThm, StRa},
as their graphs are planar and $2$-connected for any regular CW complex, and the
intersection property then implies that they are $3$-connected. Hence, all
$3$-spheres are polyhedral and we will simply refer to them as $3$-spheres.

As polyhedral spheres appear very naturally in any attempt to enumerate combinatorial types
of polytopes (see e.g.\ Grünbaum \cite[Sects.~3.3, 5.5]{Gruenbaum},
Bokowski \& Sturmfels \cite{BokowskiStu}), one must ask whether
the characterizations of sets of $f$- resp.\ flag vectors extend to polyhedral spheres,
that is, whether 
\begin{equation}\label{conj:flag}  
      \F{S}{d-1}{}\ =\ \F{P}{d}{} \qquad\text{and}\qquad
      \Fl{S}{d-1}{}\ =\ \Fl{P}{d}{}\ ?  \tag{$*$} 
\end{equation}

Steinitz's theorem \cite{SteinitzThm} \cite{StRa} in essence proves that $\SP{S}{2}{}=\SP{P}{3}{}$.
The question whether $\F{S}{d-1}{s}=\F{P}{d}{s}$, or equivalently $\Fl{S}{d-1}{s}=\Fl{P}{d}{s}$,
amounts to the “$g$-conjecture for spheres”: McMullen's conjectured answer from \cite{McMullenGConj} is “yes”;
this is known to hold for $d\le5$, as a consequence of the lower bound theorem for spheres,
proven by Barnette \cite{Barnette71}, and the sufficiency part of the $g$-Theorem for polytopes, due to Billera \& Lee \cite{BilleraLeeGTHM2}.

Up to now all available evidence with respect to  the question (\ref{conj:flag}) was positive.
There are many simplicial $3$-spheres (but also non-simplicial ones) that are non-polytopal, that is, 
not combinatorially equivalent to the boundary complex of a $4$-polytope, so
$\SP{P}{4}{s}\subsetneqq\SP{S}{3}{s}$, and consequently 
$\SP{P}{d}{s}\subsetneqq\SP{S}{d-1}{s}$ for all $d\ge4$.
Indeed, 
\emph{most} $(d-1)$-spheres are non-polytopal for $d\ge4$, as can be seen by 
comparison of Goodman \& Pollack's upper bounds on the numbers of combinatorial types of polytopes \cite{GoPo2} with the
lower bounds for spheres by 
Kalai \cite{Ka9} and Pfeifle \& Ziegler \cite{PfeifleZiegler}.
However, all the non-polytopal $3$-spheres studied so far 
turned out to have an $f$-vector (and even flag vector) that is also the $f$- (resp.~flag) vector of some $4$-polytope:
This was observed repeatedly, from the first examples 
(such as the Brückner and Barnette spheres, see e.g.\ Grünbaum \cite[Sect.~11.5]{Gruenbaum} and Ewald \cite[Sect.~III.4]{Ewal})
to the systematic enumerations of spheres with few vertices by Altshuler et al.\ (see e.g.\ 
\cite{Spheres8v2} as well as \cite[p.~96b]{Gruenbaum}).

Here we establish, for the first time, that $\Fl{S}{d-1}{}=\Fl{P}{d}{}$ does not hold in general:
Indeed, this fails for $d=4$.
For this we exhibit a specific flag vector in $\Fl{S}{3}{}\setminus\Fl{P}{4}{}$.

This flag vector belongs to a $2$-simple $2$-simplicial $3$-sphere (\emph{2s2s sphere}, for short):
A $4$-polytope or $3$-sphere is \emph{$2$-simple} if every edge lies in exactly three facets,
and \emph{$2$-simplicial} if every $2$-face has exactly three vertices.
The 2s2s property is closed under duality. 
The 2s2s $4$-polytopes were introduced by Grünbaum \cite[Sect.~4.5]{Gruenbaum}. 
This fascinating class of polytopes includes the hypersimplex and the 24-cell.
It appears naturally in the study of $\Fl{P}{4}{}$ 
(see Bayer \cite{bayer}, Ziegler \cite{Ziegler03} \cite{ZieglerUtah}, and Paffenholz \& Werner \cite{PaffenholzWerner}). 
The flag vectors of 2s2s $3$-spheres/\allowbreak $4$-polytopes are of the form
\[
(f_0,f_1,f_2,f_3;f_{02}) \ \ =\ \ (n,m,m,n;3m), 
\]
and any $3$-sphere/$4$-polytope with such a flag vector is 2s2s.
In particular, within the class 2s2s the $f$-vector determines the flag vector.
(Here and in the following we only list part of the full flag vector: All other components are determined by
the generalized Dehn--Sommerville equations of Bayer \& Billera \cite{BaBi}.)
However, there are 2s2s and non-2s2s polytopes which have the same $f$-vector:
See Paffenholz \& Ziegler \cite[Cor.~4.3]{PaffenholzZie} for examples. 
Our main result is that there is a 
specific 2s2s $3$-sphere that is non-polytopal and at the same time unique for its flag vector.
 
\begin{theorem}\label{thm:main}
    There is a unique polyhedral $3$-sphere, but no convex $4$-polytope, with flag vector given by
    \[ 
      (f_0,f_1,f_2,f_3;f_{02}) = (12,40,40,12;120).
    \]
    Thus, the set of flag vectors of $4$-polytopes is a proper subset of the
    set of flag vectors of $3$-spheres:
    \[
       \Fl{P}{4}{}\subsetneqq \Fl{S}{3}{}.
    \]
\end{theorem}

\begin{proof}
    Any $3$-sphere with the given flag vector is necessarily $2$-simplicial,
    as $f_{02}=3f_2$, and it is $2$-simple as $f_{13}=f_{02}=3f_1$.
    A $3$-sphere $W_{12}^{40}$ 
    with this flag vector was discovered in 2009 by Werner \cite[Table 7.1 left]{WernerThesis}
    in the course of his partial enumeration of shellable 2s2s $3$-spheres with $12$ vertices.
    In Section~\ref{sec:enumerate} we report about a complete enumeration
    of 2s2s $3$-spheres with at most $12$ vertices, which yields that 
    Werner's $3$-sphere $W_{12}^{40}$ is the only sphere with this flag vector.
    Finally, in Section~\ref{sec:non-polytopal} we prove that this sphere
    does not have an oriented matroid, so in particular it is not polytopal.    
\end{proof}
We will present an alternative proof of the non-polytopality of $W_{12}^{40}$ in Section \ref{sec:diagrams}. 
In Proposition \ref{prop:diagrams} we show that the sphere does not have a diagram based on the facet $F_{12}$, 
which implies that $W_{12}^{40}$ is non-polytopal.
Based on the non-polytopality, we show in Section \ref{sec:embeddability} that the sphere $W_{12}^{40}$ 
cannot be realized as a polytopal complex in any~$\R^n$.
\smallskip

We conjecture that also $\F{P}{4}{}\subsetneqq \F{S}{3}{}$, but we have not proved that.
In particular, we know of no $4$-polytope with the $f$-vector $(f_0,f_1,f_2,f_3)=(12,40,40,12)$.
As Marge Bayer has pointed out to us,  
any $3$-sphere with this $f$-vector would be ``close'' to being 2s2s, as it must satisfy
$120=3f_2\leq f_{02}\leq 130$ by \cite[Thm.~2 (3)]{bayer}. 

We would also assume that $\Fl{P}{d}{}\subsetneqq \Fl{S}{d-1}{}$
holds for all $d>4$, but again this does not seem to follow immediately from our results.

\section{2s2s 3-spheres with few vertices}\label{sec:enumerate}

The goal of this section is to report about the proof of the following result,
which includes the uniqueness claim in Theorem~\ref{thm:main}.

\begin{theorem}\label{thm:2s2s_spheres} 
        The following is a complete list of combinatorial types of 
        $2$-simple $2$-simplicial $3$-spheres with at most $12$ vertices.%
        \\[2mm] 
{\rm%
\begin{tabular}{| c | l | l | l | l |}%
\hline
$\#$vert.{}  & name & flag vector & reference & realization/polytope\\ \hline
 $5$ & $\Delta_5$      & $(5,10,10,5;30)$   &                                 & simplex\\ \hline
 $9$ & $W_9$           & $(9,26,26,9;78)$   & \cite{WernerThesis} & \cite[Thm.~4.2.2]{WernerThesis}          \\ \hline
$10$ & $W_{10}$       & $(10,30,30,10;90)$ & \cite[Sect.~4.1]{PaffenholzWerner} 
                                            &   \cite[Sect.~4.1]{PaffenholzWerner}\\
     & $\Delta_4(2)$   & $(10,30,30,10;90)$ & \cite[p.~65]{Gruenbaum} & hypersimplex \\
     & $\Delta_4(2)^*$ & $(10,30,30,10;90)$ &                 & dual of $\Delta_4(2)$\\ \hline
$11$ & $P_{11}$ & $(11,34,34,11;102)$ & \cite[Sect.~4.1]{PaffenholzWerner} 
                                           &   \cite[Sect.~4.1]{PaffenholzWerner}\\ \hline
$12$ & $W_{12}^{39}$ & $(12,39,39,12;117)$ & \cite[Tbl.~7.1 right]{WernerThesis}
                                           &   \cite[Sect.~4.2]{MiyataThesis}\\ 
     & $W_{12}^{40}$ & $(12,40,40,12;120)$ & \cite[Tbl.~7.1 left]{WernerThesis}
                                           & none: see Sect.~\ref{sec:non-polytopal} \\ \hline 
\end{tabular} 
} 
\\[1.5mm] 
All of these, except for the hypersimplex and its dual, are self-dual. 
\end{theorem}

This improves upon results of Werner \cite{WernerThesis}, who
classified the 2s2s $3$-polytopes with $f_0\le9$ vertices \cite[Thm.~7.2.13]{WernerThesis}.
He also performed a computer enumeration that produced all \emph{shellable} 2s2s 3-spheres with $f_0\le11$.
For $f_0=12$ Werner's computations remained incomplete due to constraints in computing power;
however, his incomplete enumeration produced the two spheres mentioned above.
One of these spheres, $W_{12}^{39}$, was realized as a $4$-polytope by Miyata \cite{MiyataThesis}. 

\subsection{The enumeration algorithm} 
For a $3$-sphere $S$ we define the $p$-\emph{vector} $p(S)=(p_4,p_5,\ldots)$, where $p_i$ is the number of facets of $S$ with $i$ vertices. 
For any 2s2s $3$-sphere with $f$-vector $(n,m,m,n)$ we have $p_i=0$ for $2i-4\ge n$, since a facet with $i$ vertices
is a simplicial $3$-polytope with $2i-4$ faces and thus has $2i-4$ neighboring facets. 
In particular, for $n=12$ we have $p_i = 0$ for $i > 7$.
Moreover, we have $\sum_{i\ge4} p_i=n$, and $\sum_{i\ge4}(2i-4)p_i=2m$. This yields a finite list of possible $p$-vectors for any
possible $f$-vector.
For example, for $f=(12,40,40,12)$ there are exactly $23$ potential $p$-vectors that satisfy the three restrictions. 
To enumerate all 2s2s $3$-spheres with a given number $n$ of vertices, note that
\begin{equation}
2n\le m\le \tfrac14n(n+3). \tag{M1}
\end{equation}
While the lower bound is trivial, the upper bound stems from \cite[Thm.~2 (3)]{bayer}.

We have designed and implemented an enumeration algorithm  
in order to produce, for each $p$-vector, one symmetry representative
of each set system (of vertex sets of facets) that has the given $p$-vector and is \emph{proper} in the sense
that it satisfies
\begin{compactenum}[({I}1)]
    \item the intersection of two facets contains either $0$, $1$ or $3$ vertices,\label{enum1}
    \item the intersection of three facets contains at most $2$ vertices, and\label{enum2}
    \item the intersection of four facets contains at most $1$ vertex.\label{enum3}
\end{compactenum}   
This is where we crucially use the fact that we are looking for 2s2s $3$-spheres only.
The resulting lists are then checked for being Eulerian lattices of rank $5$. 

The idea for symmetry breaking in the enumeration, and thus for avoiding to produce re-labelled versions of the same facet lists too often, 
was to fix the labelling of the vertex set of a facet of maximal size $i$, 
and then to assign step by step vertex labels to a remaining facet of maximal size. 
It turned out that in some cases even more facets could be fixed, or at least had up to re-labelling only few distinct possibilities. 
In particular this was the case whenever $p_7>0$.
 
\begin{alg}
find\_facet\_lists$(p)$\\
\textup{INPUT:} $p$-vector $(p_4,p_5,\ldots)$\\
\textup{OUTPUT:} the facet lists of all 2s2s rank $5$ Eulerian lattices with this $p$-vector up to combinatorial equivalence
\begin{compactenum}[ \rm(1)\ ]
\item  ind $= \max\{i : p_i>0\}$ 
\item  facet\_list $= \{\{0,\ldots,\textup{ind}-1\}\}$ 
\item  $p_{ind} = p_{ind}-1$ 
\item  ind $= \max\{i : p_i>0\}$ 
\item  stc $= \{\{i_0,\ldots ,i_{ind}\}: \textrm{intersection with facet\_list is proper}\}$ 
\item  for\; $F\,\in $ stc: 
\item  \hspace*{0.75cm}facet\_list $=$ facet\_list $\cup\,\{F\}$ 
\item  \hspace*{0.75cm}\textup{recursively add new facets to the list} 
\item  \hspace*{0.75cm}\textup{evaluate whenever there are enough facets in the list}
\item  \hspace*{0.55cm}facet\_list $=$ facet\_list $\setminus\,\{F\}$
\end{compactenum}
\end{alg}

After roughly two weeks of computation on standard linux workstations with altogether 45 kernels,
the algorithm had enumerated all connected 2s2s rank $5$ Eulerian lattices with up to $12$ vertices. 
This produced exactly the face lattices of the spheres listed in Theorem~\ref{thm:2s2s_spheres}, and thus proves that theorem as well as the second part of Theorem~\ref{thm:main}.

We refer to the works by Paffenholz and Werner \cite{Paffenholz2s2s}  \cite{PaffenholzWerner} \cite{WernerThesis} for information and data on 2s2s $4$-polytopes with more than $12$ vertices.

\section{Non-polytopality}\label{sec:non-polytopal}

Our approach to prove non-polytopality is via oriented matroids. 
This is a standard method for proving the non-realizability of polytopes as well as of polyhedral surfaces
(see for example Bokowski \& Sturmfels \cite{BokowskiStu}, Bokowski \cite{Boko}, Björner et al.\ \cite[Chap.~8]{orientedMatroids}),
but as far as we know this has always been applied to simplicial polytopes or surfaces, and thus in a setting of uniform oriented matroids,
with the notable exception of Bremner's software package \texttt{mpc} \cite{Bremner-mpc}, see Bokowski, Bremner \& Gévay~\cite[Sect.~7]{BBG-OMgeneration}.
(Indeed, David Bremner has confirmed the non-existence result of this section using the \texttt{nuoms} function of his package.)
Here we demonstrate that the oriented matroid method is particularly effective in an example with a non-simplicial sphere, and hence for a non-uniform oriented matroid. 

The basic approach is as follows: Any set of points $v_0,\dots,v_N\in\R^d$ leads to an orientation function
$\chi:\{v_0,v_1,\dots,v_N\}^{d+1}\rightarrow\{0,+1,-1\}$ 
by setting
\[
    \chi(v_{i_0},v_{i_1},\dots,v_{i_d})\ :=\ \textrm{sign} \det
    \begin{pmatrix} v_{i_0} & v_{i_1} & \cdots &v_{i_d} \\ 1 & 1 & \cdots & 1 \end{pmatrix}.
\]
This map is a \emph{chirotope of rank} $d+1$. In addition to the condition that its support has to be a matroid 
(which we do not use; cf.~\cite[Thm.~3.6.2]{orientedMatroids}),
this means that 
\begin{compactenum}[ (C1)] 
\item it is alternating, and 
\item it satisfies the \emph{three term Grassmann--Pl\"ucker relations}:
For any $d-1$ points $\lambda=(v_{i_0},\dots,v_{i_{d-2}})$ and four points $v_a,v_b,v_c,v_d$ the set
\[\big\{\,\chi(\lambda,v_a,v_b)\cdot\chi(\lambda,v_c,v_d),\ 
         -\chi(\lambda,v_a,v_c)\cdot\chi(\lambda,v_b,v_d),\ 
          \chi(\lambda,v_a,v_d)\cdot\chi(\lambda,v_b,v_c)\,\big\}
\]
either equals $\{0\}$ or contains $\{-1,+1\}$.
\end{compactenum}
If the points $v_0,\dots,v_N$ are supposed to be the vertices of a $d$-dimensional polytope
with a prescribed facet list $(F_1,\dots,F_n)$,
then the map must satisfy the following extra conditions: 
\begin{compactenum}[ (P1)]
\item If $v_{i_0},\dots,v_{i_d}$ are contained in a facet $F_j$, then $\chi(v_{i_0},\dots,v_{i_d})=0$.
\item \label{facet-condition}
      If $v_{i_1},\dots,v_{i_d}$ are contained in a facet $F_j$ which does not contain $v_a$ or $v_b$, then
\[ \chi(v_a,v_{i_1},\dots,v_{i_d}) = \chi(v_b,v_{i_1},\dots,v_{i_d}).
\]
\end{compactenum}

\begin{proof}[Proof of Theorem~\ref{thm:main} (the non-polytopality part)]
Werner's sphere $W_{12}^{40}$ with the flag vector $(f_0,f_1,f_2,f_3;f_{02}) = (12,40,40,12;120)$
is given by the following list of facets
(where we use the vertex labeling of \cite[Table 7.1, left]{WernerThesis}):
    \begin{center}
    \begin{minipage}{0.47\textwidth}
    $F_{\hz1}{:}\   \{v_{\hz0},v_{\hz1},v_{\hz2},v_{\hz3}\}$\\
    $F_{\hz2}{:}\   \{v_{\hz0},v_{\hz2},v_{\hz3},v_{\hz4},v_{\hz5},v_{\hz6},v_{\hz7}\}$\\
    $F_{\hz3}{:}\   \{v_{\hz0},v_{\hz1},v_{\hz3},v_{\hz4},v_{\hz8},v_{\hz9}\}$\\
    $F_{\hz4}{:}\   \{v_{\hz0},v_{\hz1},v_{\hz2},v_{\hz6},v_{\hz9},v_{10}  \}$\\
    $F_{\hz5}{:}\   \{v_{\hz0},v_{\hz4},v_{\hz7},v_{\hz8}\}$\\
    $F_{\hz6}{:}\   \{v_{\hz0},v_{\hz5},v_{\hz6},v_{10}  \}$
    \end{minipage}\quad
    \begin{minipage}{0.47\textwidth}
    $F_{\hz7}{:}\   \{v_{\hz0},v_{\hz5},v_{\hz7},v_{\hz8},v_{\hz9},v_{10}  \}$\\
    $F_{\hz8}{:}\   \{v_{\hz1},v_{\hz2},v_{\hz3},v_{\hz4},v_{10}  ,v_{11}  \}$\\
    $F_{\hz9}{:}\   \{v_{\hz2},v_{\hz5},v_{\hz6},v_{\hz8},v_{10}  ,v_{11}  \}$\\
    $F_{10}  {:}\   \{v_{\hz1},v_{\hz8},v_{\hz9},v_{10}  ,v_{11}  \}$\\
    $F_{11}  {:}\   \{v_{\hz1},v_{\hz4},v_{\hz5},v_{\hz7},v_{\hz8},v_{11}  \}$\\
    $F_{12}  {:}\   \{v_{\hz2},v_{\hz4},v_{\hz5},v_{11}  \}$
    \end{minipage}
    \end{center}
    
In order to prove non-polytopality of this sphere,
we will show that there is no chirotope compatible with its facet list.

In the sphere, $\{v_8, v_9, v_{10}\} = F_7 \cap F_{10}$ is the vertex set of a triangle $2$-face, 
so the vertices $\{v_7, v_8, v_9, v_{10}\} \subset F_7$ span a tetrahedron, 
while $v_2\notin F_7$. Thus in any realization, we have $\chi(v_7,v_8,v_{10},v_2,v_9) \neq 0$.
Thus we may fix an orientation of the realization by setting
$\chi(v_7,v_8,v_{10},v_2,v_9) := +1$.   
Starting with this, we obtain the implications 

\noindent
\begin{tikzpicture}[->,>=stealth',shorten >=1pt,auto,node distance=0.6cm,thick,
  main node/.style={font=\sffamily\bfseries},
  GPr node/.style={rectangle,fill=green!40,font=\sffamily\bfseries}]

  \node[GPr node]   (1) {$\chi(v_{\hz7},v_{\hz8},v_{10}  ,v_{\hz2},v_{\hz9})\ =\ +1$};
  \node[GPr node]   (2) [below of= 1] {$\chi(v_{\hz7},v_{\hz8},v_{10}  ,v_{11}  ,v_{\hz9}) \ =\ +1$};
  \node[GPr node]   (3) [below of= 2] {$\chi(v_{\hz7},v_{\hz8},v_{10}  ,v_{\hz4},v_{\hz9}) \ =\  +1$};
  \node[main node]  (4) [below of= 3] {$\chi(v_{\hz7},v_{\hz8},v_{10}  ,v_{\hz1},v_{\hz9}) \ =\  +1$};
  \node[main node]  (5) [below of= 4] {$\chi(v_{\hz8},v_{10}  ,v_{11}  ,v_{\hz2},v_{\hz9}) \ =\  -1$};
  \node[GPr node]   (6) [below of= 5] {$\chi(v_{\hz7},v_{\hz8},v_{10}  ,v_{11}  ,v_{\hz2}) \ =\  -1$};
  \node[GPr node]   (7) [below of= 6] {$\chi(v_{\hz4},v_{\hz8},v_{10}  ,v_{11}  ,v_{\hz2}) \ =\  -1$};
  \node[main node]  (8) [below of= 7] {$\chi(v_{\hz8},v_{10}  ,v_{\hz2},v_{\hz1},v_{\hz9}) \ =\  +1 $};
  \node[main node]  (9) [below of= 8] {$\chi(v_{\hz8},v_{10}  ,v_{\hz0},v_{\hz1},v_{\hz9}) \ =\  +1$};
  \node[main node] (10) [below of= 9] {$\chi(v_{\hz4},v_{10}  ,v_{\hz2},v_{\hz1},v_{\hz9}) \ =\  +1$};
  \node[GPr node]  (11) [below of=10] {$\chi(v_{\hz4},v_{\hz8},v_{10}  ,v_{\hz2},v_{\hz1}) \ =\  -1$};
  \node[main node] (12) [below of=11] {$\chi(v_{\hz8},v_{\hz7},v_{\hz0},v_{\hz1},v_{\hz9}) \ =\  +1$};
  \node[main node] (13) [below of=12] {$\chi(v_{\hz8},v_{\hz7},v_{\hz0},v_{\hz4},v_{\hz9}) \ =\  +1$};
  \node[main node] (14) [below of=13] {$\chi(v_{\hz8},v_{\hz7},v_{\hz0},v_{\hz4},v_{\hz5}) \ =\  +1$};
  \node[main node] (15) [below of=14] {$\chi(v_{\hz8},v_{\hz7},v_{\hz0},v_{\hz4},v_{11}) \ =\  +1$};
  \node[main node] (16) [below of=15] {$\chi(v_{\hz8},v_{\hz7},v_{\hz0},v_{\hz4},v_{\hz1}) \ =\  +1$};
  \node[GPr node]  (17) [below of=16] {$\chi(v_{\hz4},v_{\hz8},v_{10}  ,v_{\hz7},v_{11}) \ =\  +1$};
  \node[GPr node]  (18) [below of=17] {$\chi(v_{\hz4},v_{\hz8},v_{10}  ,v_{\hz7},v_{\hz1}) \ =\  +1$};
  \node[main node] (19) [below of=18] {$\chi(v_{11}  ,v_{\hz7},v_{\hz0},v_{\hz4},v_{\hz5}) \ =\  +1$};
  \node[main node] (20) [below of=19] {$\chi(v_{11}  ,v_{\hz7},v_{\hz2},v_{\hz4},v_{\hz5}) \ =\  +1$};
  \node[main node] (21) [below of=20] {$\chi(v_{11}  ,v_{\hz1},v_{\hz2},v_{\hz4},v_{\hz5}) \ =\  +1$};
  \node[main node] (22) [below of=21] {$\chi(v_{11}  ,v_{\hz1},v_{\hz2},v_{\hz4},v_{\hz8}) \ =\  +1$};
  \node[GPr node]  (23) [below of=22] {$\chi(v_{\hz4},v_{\hz8},v_{10}  ,v_{11}  ,v_{\hz1}) \ =\  +1$};

  \path[every node/.style={font=\sffamily\small}]
    (1.west) edge [bend right=60] (2.west)
    		 edge [bend right=60] (3.west)
    		 edge [bend right=60] node [left] {$F_7=\{v_{0},v_{5},v_{7},v_{8},v_{9},v_{10}\}$} (4.west)
    (2.east) edge [bend left=60] node [right] {$\{v_{1},v_{8},v_{9},v_{10},v_{11}\}=F_{10}$} (5.east)
    (5.east) edge [bend left=60] (6.east)
    		 edge [bend left=60] node [right] {$F_9=\{v_{2},v_{5},v_{6},v_{8},v_{10},v_{11}\}=F_9$} (7.east)
    (4.west) edge [bend right=60] (8.west)
    		 edge [bend right=60] node [left] {$F_{10}=\{v_{1},v_{8},v_{9},v_{10},v_{11}\}$} (9.west)
    (8.east) edge [bend left=60] node [right] {$\{v_{0},v_{1},v_{2},v_{6},v_{9},v_{10}\}=F_4$} (10.east)
    (10.east) edge [bend left=60] node [right] {$\{v_{1},v_{2},v_{3},v_{4},v_{10},v_{11}\}=F_8$} (11.east)
    (9.west) edge [bend right=60] node [left] {$F_3=\{v_{0},v_{1},v_{3},v_{4},v_{8},v_{9}\}$} (12.west)
    (12.west) edge [bend right=60] node [left] {$F_7=\{v_{0},v_{5},v_{7},v_{8},v_{9},v_{10}\}$} (13.west)
    (13.west) edge [bend right=60] (14.west)
    		 edge [bend right=60] node [left] {$F_5=\{v_{0},v_{4},v_{7},v_{8}\}$} (15.west)
    		 edge [bend right=60] (16.west)
    (15.east) edge [bend left] node [right, xshift=-1mm, yshift=1mm] {$\{v_{1},v_{4},v_{5},v_{7},v_{8},v_{11}\}=F_{11}$} (17.east)
    (16.west) edge [bend right=60] node [left] {$F_{11}=\{v_{1},v_{4},v_{5},v_{7},v_{8},v_{11}\}$} (18.west)
    (14.east) edge [bend left=60] node [right, yshift=-2mm] {$\{v_{0},v_{2},v_{3},v_{4},v_{5},v_{6},v_{7}\}=F_2$} (19.east)
    (19.east) edge [bend left=60] node [right] {$\{v_{1},v_{4},v_{5},v_{7},v_{8},v_{11}\}=F_{11}$} (20.east)
    (20.east) edge [bend left=60] node [right] {$\{v_{2},v_{4},v_{5},v_{11}\}=F_{12}$} (21.east)
    (21.east) edge [bend left=60] node [right] {$\{v_{1},v_{2},v_{3},v_{4},v_{10},v_{11}\}=F_8$} (22.east)
    (22.east) edge [bend left=60] node [right] {$\{v_{1},v_{4},v_{5},v_{7},v_{8},v_{11}\}=F_{11}$} (23.east)
    ;
\end{tikzpicture} 
\\[1mm] 
Note that this chain of arguments uses all facets except for $F_1$ and $F_6$;
all vertices occur except for $v_3$ and $v_6$.
 
Given the values of $\chi$ that we have obtained, the contradiction appears in the three term Grassmann--Pl{\"u}cker relations:\\
Let $\lambda_1=(v_7,v_8,v_{10})$, $a_1=v_4$, $b_1=v_{11}$, $c_1=v_2$, $d_1=v_9$,\\
and  $\lambda_2=(v_4,v_8,v_{10})$, $a_2=v_7$, $b_2=v_{11}$, $c_2=v_2$, $d_2=v_1$.\\
Then using the marked values of $\chi$ we get
\begin{eqnarray*}
  \big\{\,\chi(\lambda_1,a_1,b_1)\cdot\chi(\lambda_1,c_1,d_1), & 
       -\,\chi(\lambda_1,a_1,c_1)\cdot\chi(\lambda_1,b_1,d_1), &
          \chi(\lambda_1,a_1,d_1)\cdot\chi(\lambda_1,b_1,c_1)\,\big\}\ = \\ 
  \big\{  (-1)\cdot(+1), & -\,\chi(v_7,v_8,v_{10},v_4,v_2)\cdot(+1), & (+1)\cdot(-1)\big\};\\[1mm]
  \big\{\,\chi(\lambda_2,a_2,b_2)\cdot\chi(\lambda_2,c_2,d_2), &  
       -\,\chi(\lambda_2,a_2,c_2)\cdot\chi(\lambda_2,b_2,d_2), &       
          \chi(\lambda_2,a_2,d_2)\cdot\chi(\lambda_2,b_2,c_2)\,\big\}\ = \\ 
  \big\{  (+1)\cdot(-1), & -\,\chi(v_4,v_8,v_{10},v_7,v_2)\cdot(+1), & (+1)\cdot(-1)\big\}.
\end{eqnarray*}
Thus both sets contain $-1$, while by the alternating property of $\chi$ not both of them can contain $+1$.
Therefore, there is no chirotope and hence no oriented matroid for the sphere $W_{12}^{40}$,
so it is not polytopal.
\end{proof}

\section{Diagrams}\label{sec:diagrams}
A different approach to show non-polytopality of a sphere is to show that for some facet it does not have a diagram with this facet as base (see \cite[Sec.~III.4]{Ewal}, \cite[Sec.~3.3]{Gruenbaum}, or \cite[Lecture 5]{Ziegler} for details about diagrams). In \cite{Ewal} this approach was used to demonstrate non-polytopality of the Barnette sphere.

We have constructed diagrams for the sphere $W_{12}^{40}$ with facets $F_1,\ldots,F_{11}$ as bases. However, the facet $F_{12}$, which is also the only facet without a simple vertex, is special and does not serve as the base of a diagram. As an example, we show here the diagram with $F_2$ as base. Note that due to the general position argument, we can choose all coordinates to be integers. The rendering was done with JavaView \cite{javaview}.
    \begin{center}
    \begin{minipage}{0.47\textwidth}
    	\begin{eqnarray*}
		 v_{\hz0} & = & (\HZ74,\HZ20,\HZ23) \\[-3pt]
		 v_{\hz1} & = & (\HZ44,\HZ70,\HZ47) \\[-3pt]
		 v_{\hz2} & = & (\HZ30,\HZ60,  110) \\[-3pt]
		 v_{\hz3} & = & (\HZ28,  100,\HZ39) \\[-3pt]
		 v_{\hz4} & = & (\HZ44,  120,\HZ50) \\[-3pt]
		 v_{\hz5} & = & (\HZ91,\HZ88,  102) \\[-3pt]
		 v_{\hz6} & = & (\HZ44,\HZ40,  117) \\[-3pt]
		 v_{\hz7} & = & (  104,\HZ97,\HZ77) \\[-3pt]
		 v_{\hz8} & = & (\HZ83,\HZ76,\HZ58) \\[-3pt]
		 v_{\hz9} & = & (\HZ67,\HZ45,\HZ46) \\[-3pt]
		 v_{10}   & = & (\HZ61,\HZ44,\HZ83) \\[-3pt]
		 v_{11}   & = & (\HZ60,\HZ71,\HZ71) 
		\end{eqnarray*}
	\end{minipage}
	\begin{minipage}{0.47\textwidth}
		\includegraphics[scale=0.31]{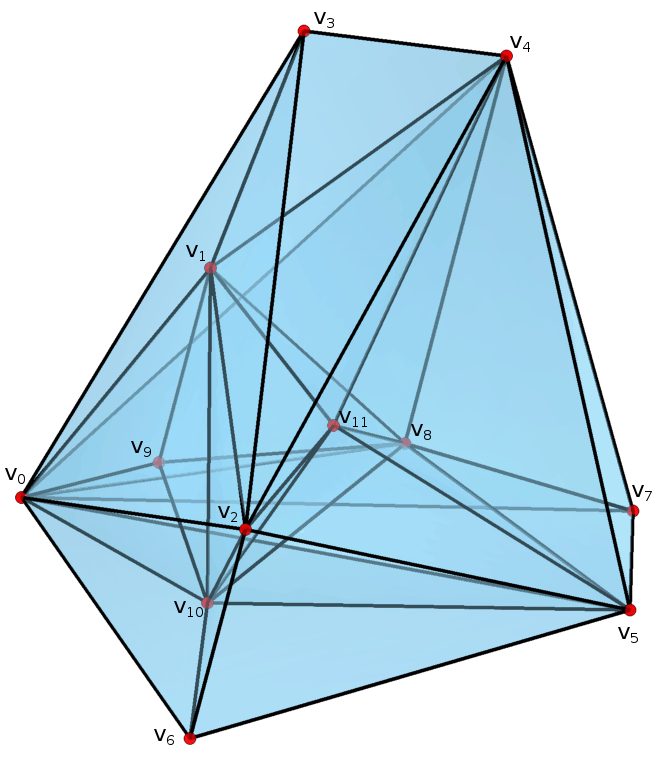}
	\end{minipage}
	\end{center}
\begin{prop}\label{prop:diagrams}
The sphere $W_{12}^{40}$ has a diagram based on every facet, but not based on $F_{12}$.
\end{prop}
\begin{proof}
Again we use oriented matroids, but this time of rank $4$, as the diagrams of $W_{12}^{40}$ live in $\R^3$. Since $W_{12}^{40}$ is $2$-simplicial, we may assume general position of the vertices, and then the oriented matroids are uniform. Now, in order to find a diagram or to prove its non-existence, we construct a partial chirotope as follows:
	\begin{itemize}
	\item Take a ridge (triangle) and some other point on a facet that contains this triangle, and
		choose a sign ($+1$ or $-1$) for the basis.
	\item Every ridge defines a plane that separates the two facets containing it. Therefore, the sign
		of the chirotope does not change when we exchange points on the same side, and flips otherwise.
		The only exception to this are the triangles on the boundary, which are the facets of the convex hull
		of the diagram.
	\item Use the Grassmann--Pl\"ucker relations to determine further entries of the partial chirotope.
	\end{itemize}
The resulting partial chirotopes for the different facets as bases give already the signs of $199$ to $328$ elements out of $\binom{12}{4}=495$, which is the size of an oriented matroid for a diagram of $W_{12}^{40}$. Using an approach recently introduced by Firsching \cite{Firsching}, we used SCIP \cite{SCIP} to find coordinates for the diagrams with bases $F_1,\ldots,F_{11}$, while in the case of the diagram with base $F_{12}$ we used backtracking to find for every oriented matroid a partial chirotope of size at least $435$ (i.e.~of roughly $87.5\%$). For all of these partial chirotopes we checked the existence of a \emph{biquadratic final polynomial} (bfp) (see \cite[Ch.~VII]{BokowskiStu} for an introduction and \cite{BoRi} for the algorithm). To find these bfps we used an implementation of the algorithm by Moritz Firsching and Arnaul Padrol. The bfps for the partial chirotopes will also be bfps for all oriented matroids that would complete this partial chirotope, so in finding bfps at this stage we could decrease the size of the search tree by a factor of $2^{60}$. There were $6098$ such partial chirotopes. To prove existence of the bfps for all these partial chirotopes, we needed roughly one week of computing time on a usual Linux workstation.
\end{proof}

\section{Embeddability}\label{sec:embeddability}
At this point we know that the sphere $W_{12}^{40}$ is non-polytopal, but is it \emph{fan-like}, \emph{embeddable}, or even \emph{star-shaped}? These properties of spheres are treated in detail by Ewald \cite[Sec.~III.5]{Ewal}. A polyhedral $d$-sphere $S$ is \emph{fan-like} if there exists a complete fan $\Sigma\subset\R^{d+1}$ together with an isomorphism that maps a face $F\in S$ to a cone $\pos F \in \Sigma$. The polyhedral sphere $S$ is \emph{embeddable} if there is a continuous injection $\varphi:|S|\rightarrow\R^n$ for some $n$ that has a continuous inverse on $\varphi(|S|)$, and which maps faces of $S$ to convex polytopes. A polyhedral $d$-sphere is \emph{star-shaped} if it has an embedding into $\R^{d+1}$ that defines a fan. Clearly it follows that a star-shaped sphere is both fan-like and embeddable. See \cite[Sec.~III.5]{Ewal} for examples that show that there are no other implications between these three properties. However, if $S$ is a simplicial sphere, then fan-like and star-shaped are equivalent, and Kalai \cite{Ka1} showed that for $d\geq 2$ every embeddable simple polyhedral $d$-sphere is polytopal.

The sphere $W_{12}^{40}$ is fan-like with the following vertex coordinates (resp. rays through the following points):
\begin{center}
\begin{tabular}{l r r r r c l r r r r}
$v_{\hz0} = ($& $9,$& $2,$& $4,$& $-1)$ && $v_{\hz6} = ($& $-5,$ & $-9,$ & $2,$ & $-16)$\\
$v_{\hz1} = ($& $-3,$& $14,$& $-16,$& $6)$ && $v_{\hz7} = ($& $-1,$ & $2,$ & $14,$ & $-1)$\\
$v_{\hz2} = ($& $-8,$& $-2,$& $-4,$& $-16)$ && $v_{\hz8} = ($& $-1,$ & $2,$ & $4,$ & $9)$\\
$v_{\hz3} = ($& $1,$& $23,$& $-1,$& $-9)$ && $v_{\hz9} = ($& $9,$ & $-7,$ & $-12,$ & $21)$\\
$v_{\hz4} = ($& $-1,$& $12,$& $4,$& $-1)$ && $v_{10} = ($& $-6,$ & $-13,$ & $-5,$ & $-1)$\\
$v_{\hz5} = ($& $-14,$& $-11,$& $27,$& $-11)$ && $v_{11} = ($& $-14,$ & $4,$ & $1,$ & $-3)$
\end{tabular}
\end{center}
These coordinates were found with the similar methods as we used for finding coordinates for the diagrams: From the combinatorics we constructed a partial chirotope of rank $5$, and then we used SCIP to find coordinates satisfying this.

However, with the methods from the previous section we were not able to decide whether $W_{12}^{40}$ is embeddable or star-shaped. For embeddability the oriented matroid approach will not work, since there are infinitely many dimensions to check. In the star-shaped case we did not find coordinates, but on the other hand we found an oriented matroid that does not have a bfp.
\begin{prop}
The sphere $W_{12}^{40}$ is not embeddable, and hence it is not star-shaped.
\end{prop}
\begin{proof}
This proof heavily depends on the structure of $W_{12}^{40}$, and establishes the following:
\begin{enumerate}
\item If $W_{12}^{40}$ is embeddable into some $\R^k$, $k\geq 4$, then it is also embeddable into $\R^4$.
\item If $W_{12}^{40}$ is embedded into $\R^4$, then this is an embedding as convex polytope.
\end{enumerate}
Let us assume that $W_{12}^{40}$ is embedded in some $\R^k$,
that is, it is realized by a polytopal complex $\Gamma$ 
(where each of the facets is realized as a convex $3$-polytope).

For (i) we note that $v_3$ is a simple vertex, that is, a vertex
that is contained in only $4$ facets, or equivalently, in only $4$ edges.
The $4$ facets are each contained in the affine span of $3$ of these
edges, so all $4$ facets are contained in the affine span of the $4$
edges. Moreover, each of the $12$ vertices of the sphere
are contained in at least one of the $4$ facets, which are $F_1,F_2,F_3$ and $F_8$.
So the complete embedded sphere is contained in this affine $4$-dimensional
subspace of $\R^k$.

For (ii) note that the argument for (i) implies that in addition
each of the facets $F_1,F_2,F_3$ and $F_8$
is also a facet of the convex hull of $\Gamma$, that is, all of $\Gamma$
is contained in a closed halfspace of the hyperplane spanned by the facet,
while only the vertices of the facet lie on the boundary hyperplane.
Exactly the same arguments are valid for the other three simple vertices of $W_{12}^{40}$
and the facets they are contained in, that is,
for $v_6$ and the facets $F_2,F_4,F_6$ and $F_9$,
for $v_7$ and the facets $F_2,F_5,F_7$ and $F_{11}$, as well as
for $v_9$ and the facets $F_3,F_4,F_7$ and $F_{10}$.
Thus we have established for all facets $F_i$ of $\Gamma$, except for
the tetrahedron $F_{12}$, that 
it is also a facet of the convex hull of $\Gamma$. As mentioned above, $F_{12}$ is special in the way that it is the only facet that does not contain a simple vertex.

To see that also $F_{12}$ is a facet of the convex hull of $\Gamma$, consider its neighbours $F_2$, $F_8$, $F_9$, and $F_{11}$. Since these are all facets of the convex hull of $\Gamma$, they cannot lie in the same hyperplane as $F_{12}$. The facet $F_2$ gives that the vertices $v_0,v_3,v_6$, and $v_7$ lie on the same side of the hyperplane $H_{12}$ spanned by $F_{12}$. Now, the facet $F_8$ contains $v_3$, the facet $F_9$ contains $v_6$, and the facet $F_{11}$ contains $v_7$, whence the points from these facets not in $F_{12}$ are also all on the same side of $H_{12}$. Therefore, we have that all points, except possibly $v_9$, lie in, say, $H_{12}^+$. Since $F_{10}$ is a bipyramid over the triangle $v_1,v_8,v_{10}$ with apexes $v_9$ and $v_{11}$, and since $v_{11}$ is in $F_{12}$, it follows $v_9\in H_{12}^+$. This completes the proof of (ii).
\end{proof}

\section*{Acknowledgements}
We are very grateful to Marge Bayer, Lou Billera, David Bremner, Moritz Firsching, Hiroyuki Miyata, Arnau Padrol, and Axel Werner for very
valuable comments, discussions, references, and data.

\begin{small}
		\newcommand\bysame{\leavevmode\vrule height 2pt depth -1.6pt width 23pt}
         
\end{small}
\end{document}